\documentclass[12pt]{article}
\usepackage{amsmath,amsthm,amssymb,latexsym,color, dsfont}
\usepackage{epsfig}
\usepackage{latexsym}

\def\abc{(\alpha, \beta)}
\def\kkk{\null\hfill $\Box $ \\}

\def\la{\left\langle}
\def\ra{\right\rangle}

\newcommand{\jm}{j\mbox{-}\min}
\newcommand{\jma}{j\mbox{-}\max}
\newcommand{\km}{k\mbox{-}\min}
\newcommand{\krm}{(k-r)\mbox{-}\min}

\newcommand{\PP}{\mathbb{P}}
\newcommand{\R}{\mathbb{R}}
\newcommand{\E}{\mathbb{E}}

\newtheorem{theo}{Theorem}[section]
\newtheorem{lem}[theo]{Lemma}
\newtheorem{cor}[theo]{Corollary}
\newtheorem{pr}[theo]{Proposition}
\newtheorem{cl}[theo]{Claim}

\theoremstyle{definition}
\newtheorem{rem}[theo]{Remark}

\expandafter\let\expandafter\oldproof\csname\string\proof\endcsname
\let\oldendproof\endproof
\renewenvironment{proof}[1][\proofname]{%
  \oldproof[\bf #1]%
}{\oldendproof}

\date{}

\title{Order statistics of vectors with dependent coordinates,
and the Karhunen--Lo\`eve basis}
\author{Alexander E. Litvak and Konstantin Tikhomirov}

\begin{document}

\maketitle

\begin{abstract}
Let $X$ be an $n$-dimensional random centered Gaussian vector with independent but not
identically distributed coordinates and let $T$ be an orthogonal transformation of $\R^n$.
We show that the random vector $Y=T(X)$ satisfies
$$
 \E\sum\limits_{j=1}^k \jm _{i\leq n}{X_{i}}^2 \leq C\E\sum\limits_{j=1}^k \jm _{i\leq n}{Y_{i}}^2
$$
for all $k\leq n$, where ``$\jm$'' denotes the $j$-th smallest component of the corresponding vector
and $C>0$ is a universal constant. This resolves (up to a multiplicative constant) an old question
of S.~Mallat and O.~Zeitouni regarding optimality of the Karhunen--Lo\`eve basis for the nonlinear
signal approximation. As a by-product we obtain some relations for order statistics of random vectors
(not only Gaussian) which are of independent interest.
\end{abstract}



\section{Introduction}

This work was motivated by the following question raised by S.~Mallat and O.~Zeitouni in 2000
(it was first posted on Zeitouni's web page and later in arxiv \cite{MZ}, see also \cite{Z}): Let $n$ be a positive integer, and given $j\leq n$ and a sequence of real numbers
$a_1,a_2,\ldots,a_n$, let $\jm _{i\leq n}a_{i}$ denote its $j$-th smallest element. Let  $X$ be an
$n$-dimensional
random Gaussian vector with independent centered coordinates (with possibly different variances).
Further, let $T$ be an orthogonal transformation of $\R^n$ and set $Y:= T(X)$.

\smallskip

{\it Is it true that for every $k\leq n$ one has
\begin{equation}\label{eq: conjecture}
\E\sum\limits_{j=1}^k \jm _{i\leq n}{X_{i}}^2
\leq \E\sum\limits_{j=1}^k \jm _{i\leq n}{Y_{i}}^2\, ?
\end{equation}}

This problem has a natural interpretation within the field of signal processing (see \cite[Chapter~IX]{M}).
Assume that a signal $Y$ is modeled as an $n$-dimensional random centered Gaussian vector
(for $n$ very large). Our goal is to approximate $Y$ by another vector which allows efficient storage and/or transmission
through narrow bandwidth channels (let us note
that this setting is distinct from the problem of signal {\it denoising} \cite[Chapter~11]{M},
in which the goal is to produce an estimator for the mean of a non-centered signal).
Let $w_i$ ($i\leq n$) be a fixed orthonormal basis
in $\R^n$, so that $Y=\sum_{i=1}^n\langle w_i,Y\rangle w_i$.
The standard approach consists in approximating $Y$ with a sparse vector with respect to that basis.
The {\it linear $m$-term approximation} of $Y$ with respect to the first $m$ basis vectors is
given by $\sum_{i=1}^m\langle w_i,Y\rangle w_i$, and the mean square error of the approximation is
$$\mathcal E_0(Y,m)=\E\sum_{i=m+1}^n\langle w_i,Y\rangle^2.$$
It is a well known fact that $\mathcal E_0(Y,m)$ is minimized when the basis $w_i$ ($i\leq n$)
is {\it the Karhunen-Lo\`eve basis} for $Y$, that is, when the random variables $\langle w_i,Y\rangle$ ($i\leq n$)
are pairwise uncorrelated, and the sequence $(\E\langle w_i,Y\rangle^2)_{i=1}^n$
is non-increasing (see, for example, \cite[Theorem~9.8]{M}).
Next, the {\it non-linear $m$-term approximation} is defined as
$\sum_{i\in\Lambda}\langle w_i,Y\rangle w_i$, where $\Lambda$ is the (random) set
of indices corresponding to $m$ largest components of $\big(|\langle w_i,Y\rangle|\big)_{i\leq n}$.
The non-linear approximation error is given by
$$\mathcal E(Y,m)=\E\sum_{i\notin\Lambda}\langle w_i,Y\rangle^2=\E\sum\limits_{j=1}^{n-m}
\jm _{i\leq n}{\langle w_i,Y\rangle}^2.$$

Now, observe that
the expression on the left hand side of \eqref{eq: conjecture}
is the mean square error
when approximating a signal $X$ with uncorrelated coordinates with respect to the standard basis
$e_1,e_2,\dots,e_n$
using its largest $n-k$ components, and that the basis $e_1,e_2,\dots,e_n$ is the
{\it Karhunen--Lo\`eve basis} for $X$.
The right hand side of \eqref{eq: conjecture} corresponds to approximation of $X$
with its $n-k$ largest components with respect to a basis $T^{-1}(e_1),T^{-1}(e_2),\dots,T^{-1}(e_n)$
(for some orthogonal transformation $T$).
Thus, \eqref{eq: conjecture} is equivalent to saying that the Karhunen--Lo\`eve basis
is optimal among all orthonormal bases in $\R^n$ regarding the nonlinear approximation
of centered Gaussian vectors.
For more information on the signal approximation we refer to \cite[Chapter~IX]{M}.

Note that the case $k=n$ is trivial. In \cite{MZ} the authors solved the problem in the special case $k=n-1$,
i.e.\ showed that
$$\E \sum\limits_{j=1}^{n-1}\jm _{i\leq n}{X_i}^2\leq \E \sum\limits_{j=1}^{n-1}\jm _{i\leq n}{Y_i}^2.$$
This corresponds to the situation when the signal is approximated by its largest one-dimensional
projection.
In this paper, we verify \eqref{eq: conjecture} up to a multiplicative constant for all $k< n$.

\begin{theo}\label{mainMZ}
Let $1\leq k <n$. Let $X$ be an $n$-dimensional centered Gaussian vector with independent coordinates and
$T$ be an orthogonal transformation of $\R^n$.
Then, setting $Y:=T(X)$, we have
$$
  \E\sum\limits_{j=1}^k \jm _{i\leq n}{X_{i}}^2
 \leq C\E\sum\limits_{j=1}^k \jm _{i\leq n}{Y_{i}}^2,
$$
where $C>0$ is a universal constant.
\end{theo}

The above theorem can be viewed as a relation between sums of {\it order statistics} of random vectors,
is which one vector has independent coordinates and the other admits dependencies.
Order statistics of random vectors are well studied objects, and numerous results regarding
their distribution are available. We refer to monography \cite{DN} for an account of developments prior to early 2000-s.
However,  in the classical setting order statistics are defined for vectors with i.i.d.\ coordinates, with some
generalizations available in the case of  independent but not identically distributed components, as well as special
kinds of dependencies (for example, exchangeable or equicorrelated coordinates; see \cite[Chapter~5]{DN}).
In our situation, we are working with coordinates which are {\it simultaneously} dependent and not equidistributed,
making their analysis more problematic.
Among recent works dealing with order statistics of vectors with dependent components let us mention
\cite{MZ,GLSW,GLSW-CRAS,GLSW-PAMS,GLSW-P,L-JFA,ALLPT,LPP}.
In particular, ideas originated in papers \cite{MZ,GLSW-CRAS,GLSW-PAMS}
play an essential role in this note.

\smallskip

The proof of Theorem~\ref{mainMZ} can be roughly divided into two (unequal) parts.
In the first part, which constitutes the novel element of this paper,
we derive a comparison inequality for sums of order statistics of two random vectors, one with independent coordinates
and the other with dependencies,
under very general assumptions on the distribution of their components. In the second part,
which essentially appeared already in \cite{MZ}, we utilize a  inequality of A.W.~Marshall and F.~Proschan \cite{MP}
(Theorem~\ref{MPth} below)
to obtain a relation between variances of coordinates of a Gaussian vector and its orthogonal transformation,
which, together with the first part, gives the statement of Theorem~\ref{mainMZ}.
The comparison inequality for order statistics is interesting on its own right, and we state it below as a separate theorem.
It holds for a class of distributions satisfying rather mild conditions (see Theorem~\ref{comparison} below), however to avoid
technical complications here, we restrict ourselves to vectors with Gaussian components.

\begin{theo}\label{t: comparison intro}
Let $p>0$, $1\leq k \leq n$ and $0<x_1\leq \ldots\leq x_n$.
Let $\xi_i$, $\eta_i$, $i\leq n$, be standard Gaussian variables and assume in addition that
 $\xi_i$, $i\leq n$, are jointly independent. Then
\begin{equation}\label{eq: comparison min}
 \mathbb E\, \,  \sum _{j=1}^{k} \, \jm_{1\leq i\leq n}|x_{i}\xi_{i}|^p  \leq
  6\, \left(C p \right)^p \,
  \mathbb E\, \,  \sum _{j=1}^{k} \, \jm_{1\leq i\leq n}|x_{i}\eta_{i}|^p  ,
\end{equation}
where $C>0$ is an absolute constant.
\end{theo}
Note that the dependencies between variables $\eta_i$, $i\leq n$, can be arbitrary;
in particular, we do not require vector $(\eta_1,\dots,\eta_n)$ to have multivariate normal distribution!

We would like to mention that in the special case $k=1$ Theorem~\ref{t: comparison intro} was previously
established in \cite{GLSW-CRAS, GLSW-PAMS}; namely, it was shown that
\begin{equation}\label{min-old}
\mathbb E\, \min_{1\leq i\leq n}|x_{i}\xi_{i}|^p  \leq \Gamma(2+p)
\mathbb E\,    \min_{1\leq i\leq n}|x_{i}\eta_{i}|^p.
\end{equation}
In turn, the last inequality can be viewed as a natural counterpart
to the well known inequality of \u{S}id\'{a}k (see \cite{Si,Gl}), asserting that
$
  \mathbb E\max_{1\leq i\leq n}|x_{i} \xi_{i}|
  \geq \mathbb E\max_{1\leq i\leq n}|x_{i} \eta_{i}| .
$
Recently, R.~van~Handel has provided an example showing that one cannot make the constant multiple
on the right hand side of \eqref{min-old} equal to $1$ even in the case $p=1$, $n=3$ \cite{MZ}.

We would also like to note that if $j$-th minima ($1\leq j\leq k$) in \eqref{eq: comparison min} are
replaced by corresponding maxima
then the expectation of the sum for independent components will be larger (up to a constant multiple), namely
\begin{equation}\label{old-orl}
  C\, \mathbb E\, \,  \sum _{j=1}^{k} \, \jma_{1\leq i\leq n}|x_{i}\xi_{i}|^p  \geq
  \mathbb E\, \,  \sum _{j=1}^{k} \, \jma_{1\leq i\leq n}|x_{i}\eta_{i}|^p,
\end{equation}
where $\jma$ denotes the $j$-th largest element of corresponding sequences and $C$ is an absolute
positive constant. We refer to Theorem~4 in \cite{GLSW} (see also Theorem~2.4 in \cite{GLSW-Bull}),
where this result was proved in a more general setting involving
arbitrary Orlicz norms (note that the sum $\sum _{j=1}^{k}\jma_{i\leq n} |z_{i}|$ is equivalent to an
Orlicz norm of the sequence $(z_i)_{i\leq n}$).
In \cite{MS} this result was further extended to an even wider class of norms. We would like to emphasize that
although
$$
 \mathbb E\, \,  \sum _{j=1}^{n} \, \jma_{1\leq i\leq n}|x_{i}\eta_{i}|^p  =
  \mathbb E\, \,  \sum _{j=1}^{n} \, \jm_{1\leq i\leq n}|x_{i}\eta_{i}|^p =
  \mathbb E\, \,  \sum _{j=1}^{n} \, |x_{i}\eta_{i}|^p =
  \Big( \sum _{j=1}^{n} \, |x_{i}|^p \Big) \, \mathbb E\,\,|\eta_{1}|^p ,
$$
the estimates (\ref{eq: comparison min}) and  (\ref{old-orl}) are incomparable -- none of them
implies the other one.

\medskip

One of important ingredients in the proof of Theorem~\ref{t: comparison intro} is
a statement which provides optimal estimates for sums of the smallest order
statistics in case of independent components (see Theorem~\ref{mainn}).
The proof is based on using special functionals which were previously employed
in papers \cite{GLSW-CRAS,GLSW-PAMS,GLSW-P}.

Another novel element is an argument for working with dependent components
(see Theorem~\ref{dep-median}).
Absence of such a tool in preceding works \cite{GLSW-CRAS,GLSW-PAMS}
was a major obstacle to proving the Mallat--Zeitouni conjecture, even
up to a multiple depending on $k$.
The proof of Theorem~\ref{dep-median} is essentially reduced to considering
uniformly bounded dependent variables.


The paper is organized as follows. In Section~\ref{secone}, we fix notation and provide auxiliary
statements. Additionally, we introduce several special conditions on distributions which
are assumed (in various combinations) in our main statements.
Section~\ref{min-bounds} contains some known results on individual order statistics, which we use later in the paper.
For the sake of completeness we provide the proofs, but we postpone them to Section~\ref{aux-res-pr}.
Section~\ref{indordstat} provides new bounds for
individual order statistics playing a crucial role in the proof of the main results. The next two sections
are devoted to proving Theorems~\ref{mainMZ} and~\ref{t: comparison intro}.
In Section~\ref{s: nonlinear efficiency} we briefly discuss efficiency of the nonlinear
approximation based on the largest projections, compared to the linear approximation.

\section{Notation and preliminaries}
\label{secone}

Given a subset $A\subset \mathbb N$, we denote its cardinality by $|A|$.
Next, for a natural number $n$ and a set $E\subset \{1,2,\dots,n\}$,
we denote by $E^c$ the complement of $E$ inside $\{1,2,\dots,n\}$.
Similarly, for an event $A$ we denote by $A^c$ the complement of the event.
Further, we say that a collection of sets $(A_j)_{j\leq k}$ is a partition of $\{1,2,\dots,n\}$ if
each $A_j$ is non-empty, the sets are pairwise disjoint and their union is $\{1,2,\dots,n\}$.
The canonical Euclidean norm and the canonical inner product in
$\R^n$ will be denoted by $|\cdot|$ and $\la \cdot, \cdot \ra $, respectively.
We adopt the conventions $1/0=\infty$ and $1/\infty=0$ throughout the text.
For a given sequence of real numbers $a_1,a_2\dots,a_n$, we denote
its $k$-th smallest element by $\km _{1\leq i\leq n}a_i$. In particular,
$1\mbox{-}\min_{1\leq i\leq n}a_{i}=\min_{1\leq i\leq n}a_{i}$,
and $(\km _{1\leq i\leq n}a_{i} )_{k=1}^n$ is the
non-decreasing rearrangement of the sequence $(a_i)_{i=1}^n$.
As usual, we use the abbreviation cdf for the cumulative distribution function
(that is, given a random variable $\xi$, the cdf of $\xi$ is $F(t)=\PP(\xi \leq t)$).

\medskip

Next, we group together a few combinatorial results which provide basic tools for estimating order statistics in next sections.
Let us start with the following simple property of $\km _{1\leq i\leq n}a_{i}$ which holds for every real sequence  $(a_i)_{i=1}^n$:
{\it For every partition $(A_j)_{j\leq k}$ of $\{1, 2, \dots, n\}$ one has}
\begin{equation}\label{partition}
\sum_{j=1}^k  \jm_{1\leq i\leq n}a_{i} \leq  \sum_{j=1}^k  \min_{i\in A_j} a_i.
\end{equation}
The next statement is a classical inequality for symmetric means:

\begin{theo}[{C. Maclaurin, see \cite[Theorem~52]{HLP}}]
Let $1\leq \ell\leq n$ and let $a_{1},\ldots, a_n$ be nonnegative
real numbers. Then
$$
\sum_{A\subset \{1, 2, \ldots, n\} \atop |A|=\ell}\prod_{i\in A}a_{i}
\leq {n\choose \ell}\left(\frac{1}{n}\sum_{i=1}^{n}a_{i} \right)^{\ell}.
$$
\end{theo}
In \cite{GLSW-PAMS} it was shown that the above statement, together with Stirling's
formula, implies

\begin{cor}\label{agmean}
Let $1\leq k\leq n$. Let $a_{1},\ldots, a_n$ be nonnegative
real numbers and assume that
$$0<  a:= \frac{e}{k}\, \sum _{i=1}^n a_i < 1.$$
Then
$$\sum _{\ell=k}^n\sum_{A\subset \{1, 2, \ldots, n\} \atop |A|=\ell}\prod_{i\in A}a_{i}
<\frac{1}{\sqrt{2 \pi k}} \, \,  \frac{a^k}{1-a}.$$
\end{cor}
The following statement was essentially obtained in \cite{GLSW-PAMS} (cf. Lemma~4
there). We reproduce the argument in Section~\ref{aux-res-pr} for reader's convenience.

\begin{lem}\label{aboutmin}
Let $1\leq k\leq n$ and let $(a_{i})_{i=1}^n$ be a non-increasing
sequence of positive real numbers. For each $j\leq n$, set  $b_j:=\sum_{i=j}^{n} a_{i}$
and let $m\leq k$ be the smallest integer such that
$$a_m \leq \frac{b_m}{k+1-m}.$$
Then there exists a partition
$(A_j)_{j\leq k}$ of $\{1, 2, \dots, n\}$
such that $A_j=\{j\}$ for $j<m$ and for every $j\geq m$ we have
$$
\sum_{i\in A_j} a_{i} \geq \frac{b_m}{2(k+1-m)}.
$$
\end{lem}

\begin{rem}\label{new-rem-25}
In fact, as one can see from the proof below, the sets $A_j$ can be chosen as intervals, i.e.\
$A_j = \left\{ i\leq n \,:\, n_{j-1} < i \leq {n_j} \right\}$,
$j\leq k$, for some sequence $0=n_0 <  n_1 < \ldots < n_k =n$. Moreover, with the partition
used in the proof we also have
$$
\min_{1\leq \ell \leq k} \, \sum_{i\in A_\ell} a_i \geq \tfrac{1}{2} \
\min_{1\leq j \leq k} \, \frac{1}{k+1-j} \ \sum_{i=j}^{n} a_{i} .
$$
\end{rem}

\bigskip

Next, we introduce several conditions on distributions of random variables.
Let $\alpha >0$ and $\beta >0$
be parameters. We say that a random variable $\xi$ satisfies the
{\it $\alpha$-condition} if
\begin{equation} \label{gdistone}
 \PP  \left(|\xi | \leq t \right) \leq  \alpha t \quad\mbox{ for every }\quad t\geq 0
\end{equation}
and $\xi$ satisfies
the {\it $\beta$-condition} if
\begin{equation} \label{gdisttwo}
 \PP  \left(|\xi | > t \right) \leq e^{-\beta t} \quad\mbox{ for every }\quad t\geq 0.
\end{equation}
If both \eqref{gdistone} and~\eqref{gdisttwo} hold then we say that $\xi$ satisfies the {\it $\abc$-condition}.
Note that in
this case we necessarily have $\alpha  t +e^{- \beta t} \geq 1$ for all $t \geq 0$,
which can be true only for $\alpha \geq \beta$.
In \cite{GLSW-PAMS} it was shown that for any $q\geq 1$, a non-negative random variable $\xi$
with the density function $p(s) = c_q \exp{(-s^q)}$ ($s\geq 0$), where
$c_q:= 1/ \Gamma (1+1/q)$, satisfies \eqref{gdistone}
and \eqref{gdisttwo} with parameters $\alpha = \beta = c_q$.
In particular, for $q=2$ we get a Gaussian random variable $\mathcal{N}(0,1/2)$, and $\alpha = \beta = c_2=2/\sqrt{\pi}$.
This easily implies that the standard Gaussian variable satisfies \eqref{gdistone} and~\eqref{gdisttwo}
with $\alpha = \beta = \sqrt{2/\pi}$.
Note also that for $q=1$ we have an
exponentially distributed random variable satisfying the $\abc$-condition with $\alpha = \beta =1$.
Finally, it is not difficult to check that any centered log-concave random variable satisfies the $\abc$-condition
for some $\alpha$ and $\beta$.

We will employ one more condition on a cdf $F$ of a non-negative random variable:
\begin{equation}\label{cdf-decay}
\mbox{{\it there exist $\, \, \delta\in (0,1)$, $\, A>1\, $ such that
$\, F(t) \geq 2 F(t/A)\, \, $ whenever $\, F(t)\leq \delta$.}}
\end{equation}
Note that the multiple ``$2$'' on the right hand side of \eqref{cdf-decay} can be
replaced with any number $a>1$,
at expense of increasing $A$ and decreasing $\delta$.

\medskip

Finally, we state the following result of Marshall and Proschan, which will be used in Section~\ref{proof}:
\begin{theo}[{\cite{MP}}]\label{MPth}
Let $\xi_1, \ldots, \xi_n$ be interchangeable random variables (that is,
with the joint distribution invariant under permutations of arguments).
Let $(a_i)_{i\leq n}$ and $(b_i)_{i\leq n}$ be non-negative non-increasing
sequences such that for every $\ell\leq n$
$$
  \sum _{i=1}^\ell a_i \geq \sum _{i=1}^\ell b_i \quad \mbox{ and } \quad
    \sum _{i=1}^n a_i = \sum _{i=1}^n b_i.
$$
Let $\varphi$ be a continuous convex function symmetric in its $n$ arguments.
Then
$$
   \E \, \varphi(a_1 \xi_1, \ldots, a_n \xi_n)\geq \E\, \varphi(b_1 \xi_1, \ldots, b_n \xi_n).
$$
\end{theo}

\section{Known bounds for individual order statistics}
\label{min-bounds}

In this section we recall some of results from papers \cite{GLSW-CRAS, GLSW-PAMS}
concerning order statistics. For the sake of completeness we provide their proofs in
Section~\ref{aux-res-pr}.

\begin{lem}\label{lemmin-lower}
Let $\alpha >0$ and $p>0$. Let $0<x_1\leq x_2\leq \ldots \leq x_n$ be real numbers and let
$\xi _1, \dots, \xi_n$ be (possibly dependent) random variables
satisfying the $\alpha$-condition. Finally, set $b:= \sum_{i=1}^n 1/x_i$.
Then for every $t>0$ we have
$$
\PP\big\{ \min_{1\leq i\leq n}|x_{i}\xi_{i}|\leq t\big\} \leq \alpha b\,t.
$$
In particular,
$$
\mbox{\rm Med}\big(\min_{1\leq i\leq n}|x_{i}\xi_{i}|^p \big)
\geq \frac{1}{2^p\alpha^p b^p} \quad \mbox{ and } \quad
\E \min_{1\leq i\leq n} |x_{i}\xi _{i}|^p \geq \frac{1}{(1+p) \alpha ^p \, b^p}.
$$
\end{lem}

\begin{lem}\label{lemmin-upper}
Let $\beta >0$ and $p>0$. Let $0<x_1\leq x_2\leq \ldots \leq x_n$ be real numbers and let
$\xi _1, \dots, \xi_n$ be independent random variables satisfying the
$\beta$-condition. Set $b:= \sum_{i=1}^n 1/x_i$. Then for every $t>0$ we have
$$\PP\big\{ \min_{1\leq i\leq n}|x_{i}\xi _{i}|> t \big\} \leq  e^{-\beta b\, t}.
$$
In particular,
$$
\mbox{\rm Med}\big(\min_{1\leq i\leq n}|x_{i}\xi_{i}|^p \big) \leq \frac{(\ln 2)^p}{\beta^p\, b^p} \quad \mbox{ and } \quad
\E \min_{1\leq i\leq n} |x_{i}\xi _{i}|^p \leq \frac{\Gamma(1+p)}{\beta ^p \, b^p}.
$$
\end{lem}

An immediate consequence of the above lemmas is the following statement.

\begin{cor}\label{malzeit}
Let $p>0$. Let $(x_i)_{i=1}^n$ be a sequence of real numbers and
$f_1, \dots, f_n$,  $\xi _1, \dots, \xi_n$ be random variables
satisfying the $\abc$-condition for some $\alpha,\beta>0$. Assume additionally that the $\xi_i$'s are jointly
independent. Then
$$
\mathbb E\min_{1\leq i\leq n}|x_{i} \xi_{i}|^{p} \leq
\Gamma(2+p) \, \alpha ^p \, \beta^{-p} \,
\mathbb E\min_{1\leq i\leq n}|x_{i} f_{i}|^{p} .
$$
In particular, if $f_1, \dots, f_n$, $\xi _1, \dots, \xi_n$ are
$N(0,1)$ Gaussian random variables, then
$$
\mathbb E\min_{1\leq i\leq n}|x_{i} \xi_{i}|^{p} \leq
\Gamma(2+p) \, \mathbb E\min_{1\leq i\leq n}|x_{i} f_{i}|^{p} .
$$
\end{cor}

\medskip

The next lemma deals with order statistics other than the smallest one:

\begin{lem}\label{forthtwo}
Let $\alpha >0$, $p >0$ and $1\leq k\leq n$.
Further, let $0<x_1\leq x_2\leq \ldots \leq x_n$ be real numbers and let
$\xi _1, \dots, \xi_n$ be independent random variables satisfying
the $\alpha$-condition. Set $b := \sum_{i=1}^n 1/x_i$, $a:=\alpha e b/k$.
Then for every $0<t<1/a$ we have
\begin{equation*} \label{newkminn}
  \PP\big\{\km_{1\leq i\leq n} |x_{i}
  \xi _{i}|\leq t  \big\} \leq \frac{1}{\sqrt{2
  \pi k}} \, \, \frac{(a t)^k}{1-a t}
\end{equation*}
and
$$
   \frac{1}{2^{1/p}\, 4\, \alpha}\,  \max_{1 \leq j \leq k}\ \frac{k+1-j}{
   \sum_{i=j}^n 1/x_i} \leq
   \left(\E\, \km_{1\leq i\leq n} |x_{i}\xi _{i}|^p \right)^{1/p} .
$$
\end{lem}

\begin{rem}\label{comp-ind-sch}
Using Lemmas~\ref{lemmin-upper}, \ref{aboutmin} (with Remark~\ref{new-rem-25})
and ideas similar to ones used in the proof of Theorem~\ref{mainn} below,
it was shown in \cite{GLSW-PAMS} that for variables satisfying the $\beta$-condition we have
$$
  \left( \mathbb E \, \,
  \km_{1\leq i\leq n} |x_{i} \xi_{i}| ^p \right)^{1/p} \leq
  C (p, k)   \, \beta ^{-1} \, \max_{1 \leq j \leq k} \
  \frac {k+1-j}{\sum_{i=j}^n 1/x_i } ,
$$
where $C(p, k) := C \max\{ p,  \ln (k+1) \}$,
and $C$ is an absolute positive constant. Moreover, in \cite{GLSW-P}
it was shown that the expectation above is equivalent to some Orlicz
norm (up to a factor logarithmic in $k$).
\end{rem}

\medskip

\section{New bounds for individual order statistics}
\label{indordstat}

Let $\xi$ be a real-valued random variable and let $F=F_{\xi}$ be its cdf. Let $r\in [0,1]$. By
$q(r)=q_F(r)=q_{\xi} (r)$ we denote a quantile of order $r$, that is a number satisfying
$$
   \PP\left\{\xi<q(r)\right\} \leq r \quad \mbox{ and } \quad
   \PP\left\{\xi \leq q(r)\right\} \geq r
$$
(note that in general $q(r)$ is not uniquely defined).
The following claim provides simple lower bounds on quantiles for a large class of random variables.

\begin{cl} \label{quan-k-min}
Let $1\leq k \leq n$ and $0<x_1\leq \ldots\leq x_n$. For each $j\leq n$, set $b_j:=\sum _{i=j}^n 1/x_i$.
Further, let $\xi_i$, $i\leq n$, be (possibly dependent) random variables satisfying the $\alpha$-condition
for some $\alpha>0$, and
for every $i\leq n$ let $F_i$ be the cdf of $|x_i \xi_i|$. Denote
$$
   F:=\frac{1}{n} \, \sum _{i=1}^n F_i \quad \mbox{ and } \quad q:=q_F\left(\frac{k-1/2}{n}\right).
$$
Then
$$
   q\geq \frac{1}{2\alpha}\,   \max _{1\leq j\leq k}  \frac{k-j+1}{ b_j }.
$$
\end{cl}

\begin{proof}
By the above definitions, for every $j\leq k$ we have
$$
  k-1/2 \leq \sum _{i=1}^n F_i (q) \leq j-1 + \sum _{i=j}^n  \alpha q/x_i
 = j-1 + \alpha q b_j,
$$
which implies the result.
\end{proof}

\begin{rem}
It is not difficult to check that when all $\xi_i$'s are uniformly distributed on $[0,1]$, we have
$$
q_F\left(\frac{k-1/2}{n}\right) =  \max _{1\leq j\leq k}  \frac{k-j+1/2}{ b_j }.
$$
\end{rem}

\medskip

The next lemma provides lower estimates for order statistics of possibly
dependent random variables via quantiles of their truncations.

\begin{lem}\label{dep-var}
Let $\delta\in (0,1)$, $A>1$ and $x_1, \ldots, x_n>0$.
Let $\xi_i$, $i\leq n$, be (possibly dependent) random variables satisfying
condition \eqref{cdf-decay} with parameters $\delta$ and $A$. Further, define
$$
t_0 := \min _{i\leq n} \sup\{t>0 \, : \, F_{|\xi_i|} (t)\leq \delta \} \quad \mbox{ and } \quad
\eta _i := \min(|\xi _i|, t_0) , \, \, i\leq n.
$$
For every $i\leq n$, we let $F_i$ be the cdf of $x_i \eta_i$. Define
$$F=\frac{1}{n} \, \sum _{i=1}^n F_i.$$
Then
$$
 \mbox{\rm Med} \Big( \km_{1\leq i\leq n}|x_{i}\xi_{i}|\Big) \geq \frac{q_F\left(\frac{k-1/2}{n}\right)}{A}.
$$
\end{lem}

\begin{rem} It may seem natural to obtain a bound for the median in terms of
the ``averaged'' cdf with respect to the original variables $\xi_i$ and not their truncations $\eta_i$.
The following example (cf. Example~12 in \cite{GLSW-PAMS})
shows that in fact the truncation is essential.
Consider independent standard Gaussian random variables $g_1, g_2, \dots, g_n$ and
let $\xi_i:=g_1$ for $i\leq n$. Clearly, these random variables satisfy condition \eqref{cdf-decay} with some $A$ and $\delta$.
Let $x_1=\ldots=x_k=1$ and $x_{k+1}=\ldots = x_n=n^2$. Then a direct computation shows
that for $G:=\frac{1}{n}\sum_{i=1}^n F_{|x_ig_i|}$ we have
$$
  q_G\left(\frac{k-1/2}{n}\right) \approx \mbox{\rm Med} \left( \km_{1\leq i\leq n}|x_{i}g_{i}|\right)\approx \sqrt{\ln k}
$$
while
$$
 \mbox{\rm Med} \left( \km_{1\leq i\leq n}|x_{i}\xi_{i}|\right)\approx const.
$$
\end{rem}

\begin{proof}[Proof of Lemma~\ref{dep-var}.]
Clearly,
$F_{\eta_i}(t) = F_{|\xi_i|} (t) \leq \delta$ for $t< t_0$ and $F_{\eta_i}(t) =1$ for $t\geq t_0$, and
that $F_{\eta_i}$ also satisfies condition \eqref{cdf-decay} with parameters $\delta$ and $A$.
Fix some positive $s<q_F\left(\frac{k-1/2}{n}\right)/A$ and denote
$$
   I:=\{i\leq n\, : \, F_i(As) = 1 \}.
$$
By the choice of $s$ we have $\sum _{i\leq n} F_i(As)=n F(As)<k-1/2$, whence $|I|<k$. Note also that
for every $i\not\in I$ we have $F_i(As)\leq \delta$.
Thus, applying condition \eqref{cdf-decay}, we get
\begin{align*}
\E |\{i\in I^c\, : \, x_i \eta_i < s \}| &= \E\sum_{i\in I^c}\chi_{\{x_i \eta_i < s\}}
\leq \sum_{i\in I^c} F_i(s)\\
&\leq \frac{1}{2} \sum_{i\in I^c} F_i(As)
=\frac{n F(As)- |I|}{2}
< \frac{k- |I|}{2} .
\end{align*}
By Markov's inequality this implies
$$
  \PP \big(  |\{i\in I^c\, : \, x_i \eta_i < s \}| \geq k- |I| \big) \leq \frac{1}{2},
$$
whence
$$
  \PP \big(  |\{i\leq n \,  : \, x_i \eta_i < s \}| \geq k \big) \leq \frac{1}{2}.
$$
Since the event $\{\km_{1\leq i\leq n}x_{i}\eta_{i} \geq s\}$ coincides with the event
$\{|\{i\leq n \,  : \, x_i \eta_i < s \}| < k\}$, we obtain
$$
  \PP \big(  \km_{1\leq i\leq n}x_{i}\eta_{i} \geq s \big) \geq \frac{1}{2},
$$
that is
$$
  \mbox{\rm Med} \left( \km_{1\leq i\leq n}|x_{i}\xi_{i}| \right) \geq
  \mbox{\rm Med} \left( \km_{1\leq i\leq n}|x_{i}\eta_{i}| \right) \geq s .
$$
Since $s$ was an arbitrary number smaller that $q_F\left(\frac{k-1/2}{n}\right)/A$, the proof is complete.
\end{proof}

Now, let us formulate a new theorem on order statistics,
which essentially states that the lower bound for expectation in Lemma~\ref{forthtwo}
does not require independence.

\begin{theo}\label{dep-median}
Let $\alpha>0$, $\delta\in (0,1)$, $A>1$,
$1\leq k \leq n$ and $0<x_1\leq \ldots\leq x_n$. For each $j\leq n$, we set $b_j:=\sum _{i=j}^n 1/x_i$.
Further, let $\xi_i$, $i\leq n$, be (possibly dependent) random variables satisfying the $\alpha$-condition
and condition \eqref{cdf-decay} with parameters $\delta$ and $A$. Then
$$
 \mbox{\rm Med} \left( \km_{1\leq i\leq n}|x_{i}\xi_{i}|\right) \geq \frac{\delta}{2 A \alpha}\,
  \max _{1\leq j\leq k}  \frac{k-j+1}{ b_j }.
$$
\end{theo}

\begin{proof}
Let the number $t_0$ and random variables $\eta_i$, $i\leq n$ be defined as in Lemma~\ref{dep-var}.
Note that the $\alpha$-condition on $\xi_i$'s implies
$F_{|\xi_i|} (t) \leq \delta$, $i\leq n$, for $t\leq \delta/\alpha$.
Hence, $t_0\geq \delta/\alpha$. Thus, for every $i\leq n$ we have
$F_{\eta_i} (t) \leq \alpha t$ whenever $t<t_0$ and
 $F_{\eta_i} (t) = 1  \leq (\alpha/\delta) t$ otherwise. In other words,
the random variables $\eta_i$'s satisfy condition~\eqref{gdistone} with
$\alpha/\delta$ replacing $\alpha$. Combining Lemma~\ref{dep-var} with
Claim~\ref{quan-k-min}, applied to $\eta_i$'s, we obtain the result.
\end{proof}

In view of Remark~\ref{comp-ind-sch}, Theorem~\ref{dep-median} has the following consequence.

\begin{cor}\label{cor-ind-sch}
Under the conditions of Theorem~\ref{dep-median}, assuming that independent random variables 
$\eta_1, \ldots, \eta _n$ satisfy $\beta$-condition with some $\beta >0$, one has for every 
$p>0$, 
$$
  \left( \E \, \km_{1\leq i\leq n}|x_{i}\eta _{i}|^p\right)^{1/p} \leq 
    \frac{C\, 2^{1/p}\, A\, \alpha}{\beta\, \delta}\, \max\{p, \ln (k+1)\} \, 
  \left( \E\,  \km_{1\leq i\leq n}|x_{i}\xi_{i}|^p\right)^{1/p} ,
$$
where $C$ is an absolute positive constant.
\end{cor}

\begin{rem}
The logarithmic factor in Corollary~\ref{cor-ind-sch} cannot be removed as the following example shows. 
Let $\xi$ be a positive exponential random variable, that is a random variable with the density
function $f(t)= e^{-t}$ for $t\geq 0$ and $f(t)=0$ for $t<0$. Let $\eta_i$, $i\leq n$, be independent 
copies of $\xi$ and $\xi_1=\dots =\xi_n=\xi$. 
Let $x_1=\dots =x_k=1$ and $x_{k+1}=\dots =x_n=n^2$. Then 
$$
  \E \,  \km_{1\leq i\leq n}|x_i \eta _{i}| \approx \E  \max_{1\leq i\leq k}| \eta _{i}| \approx \ln (k+1), \quad 
  \mbox{ while } \quad 
  \E\,  \km_{1\leq i\leq n}|x_{i}\xi_{i}| = \E  |\xi | = 1.
$$
\end{rem}

\section{Bounds for sums of order statistics}
\label{sectri}
\label{sectiontwo}

\begin{theo}\label{mainn}
Let $p>0$ and let $\xi_i$, $i\leq n$, be independent random variables satisfying the $\abc$-condition
for some $\alpha,\beta>0$.
Let $0<x_1\leq \ldots\leq x_n$. For each $j\leq n$, set $b_j:=\sum _{i=j}^n 1/x_i$. Then for every $k\leq n$ we have
$$
\frac{1}{2} \left(\frac{1}{16\alpha}\right)^{p}\, \sum _{j=1}^{k}  \frac{(k-j+1)^{p}}{ {b_j}^p}
\leq
\mathbb E\, \,  \sum _{j=1}^{k} \, \jm_{1\leq i\leq n}|x_{i}\xi_{i}|^p  \leq
W(\beta,p) \sum _{j=1}^{k}  \frac{(k-j+1)^{p}}{ {b_j}^p},
$$
where $W(\beta,p):= \beta^{-p} \, \Gamma(1+p)\big(1 + 2\cdot 4^p\big)$.
\end{theo}

\smallskip

\begin{rem}
The upper bound can be replaced with a slightly stronger equivalent estimate
$$
 \beta^{-p} \, \Gamma(1+p)\sum _{j=1}^{m-1}   {x_{j} }^{p}
  +  2^p\beta^{-p} \, \Gamma(1+p)  \,  \frac{ (k-m+1)^{1+p}}{ {b_m}^p },
$$
where $m\leq k$ is the smallest positive integer such that
$$
\frac{1}{x_m} \leq \frac{b_m}{k+1-m}
$$
(see the proof below).
\end{rem}

\medskip

We will need the following calculus lemma.

\begin{lem}\label{low-est}
Let $p>0$ and $0<x_1\leq \ldots\leq x_n$. For $j\leq n$, set $b_j:=\sum _{i=j}^n 1/x_i$.
Then
$$
 4^p \, \sum _{\ell=1}^{k} \, \max_{1\leq j\leq \ell} \frac{(\ell -j+1)^p}{{b_j}^p}  \geq
   \sum _{j=1}^{k}  \frac{(k-j+1)^{p}}{ {b_j}^p}  \geq
   2^{-1-p}  \max _{1\leq j\leq k}  \frac{ (k-j+1)^{1+p}}{ {b_j}^p } .
$$
\end{lem}
\begin{proof}
For some fixed $1\leq s\leq k$ let $\ell: = \lceil (k+s)/2\rceil$. Then
$$
   \frac{k-s+1}{b_{s}} \leq 2\, \frac{\ell -s+1}{b_{s}} \leq
   2\, \max_{1\leq j\leq \ell} \frac{\ell-j+1}{b_j}.
$$
Hence,
$$
  \sum _{s=1}^{k}  \frac{(k-s+1)^{p}}{ {b_s}^p}  \leq 4^p\,
   \sum _{\ell=\lceil k/2\rceil}^{k} \max_{1\leq j\leq \ell} \frac{(\ell-j+1)^p}{{b_j}^p},
$$
which implies the left hand side inequality.

Now let $s\leq k$ be such that
$$
   \max _{1\leq j\leq k}  \frac{ (k-j+1)^{1+p}}{ {b_j}^p } =    \frac{ (k-s+1)^{1+p}}{ {b_s}^p }.
$$
Let $t:=\lfloor (k+s+1)/2\rfloor$.  Then
$$
  \sum _{j=1}^{k}\, \frac{(k-j +1)^p}{{b_j}^p} \geq
  \sum _{j=s}^{t}  \frac{(k-t +1)^p}{{b_s}^p} \geq (t-s+1) \frac{(k - s +1)^{p}}{(2 b_s)^p}
 \geq \frac{(k-s +1)^{1+p}}{ 2^{1+p}\, {b_s}^p},
$$
which completes the proof.
\end{proof}

\medskip

\begin{proof}[Proof of Theorem~\ref{mainn}]
For the lower bound, by Lemma~\ref{forthtwo} we have
$$\mathbb E\, \,  \sum _{j=1}^{k} \, \jm_{1\leq i\leq n}|x_{i}\xi_{i}|^p
\geq \frac{1}{2\, (4\, \alpha)^p}\sum_{\ell=1}^k \,  \max_{1 \leq j \leq \ell}\ \frac{(\ell-j+1)^p}{
   (\sum_{i=j}^n 1/x_i)^p},$$
and it remains to apply Lemma~\ref{low-est}
(alternatively, under slightly modified assumptions on random variables
we could use Theorem~\ref{dep-median}).
Let us prove the upper bound.
Let $B_p:=\beta^{-p} \, \Gamma(1+p)$.
Let the integer $m$ and the partition $(A_j)_{j\leq k}$ be given by
Lemma~\ref{aboutmin} applied to the sequence $(a_i)_{i\leq n}: = (1/x_i)_{i\leq n}$.
Using
\eqref{partition} and Lemma~\ref{lemmin-upper}, we get
$$
\mathbb E\, \,  \sum _{j=1}^{k} \, \jm_{1\leq i\leq n}|x_{i}\xi_{i}|^p  \leq
\mathbb E\, \, \sum _{j=1}^{k}  \min_{i\in A_j} |x_{i} \xi_{i}| ^p
\leq B_p\,\sum _{j=1}^{m-1}   {x_{j} }^{p} + B_p \,
\sum _{j=m}^{k} \bigg( \sum_{i\in A_j} 1/x_{i} \bigg) ^{-p}.
$$
Next, note that by the choice of the partition $(A_j)_{j\leq k}$ we have
$$
\sum _{j=m}^{k} \bigg( \sum_{i\in A_j} 1/x_{i} \bigg) ^{-p}
\leq 2^p\frac{
(k-m+1)^{1+p}}{ {b_m}^p}\leq 2^p \, \max _{1\leq j\leq k}  \frac{ (k-j+1)^{1+p}}{ {b_j}^p }.$$
Further, applying the definition of $m$ to numbers $x_j$, $j<m$, we obtain
$$\frac{1}{x_j}> \frac{b_j}{k+1-j},\quad j<m,$$
whence
$$\sum _{j=1}^{m-1}    {x_{j} }^{p}\leq \sum _{j=1}^{m-1}  \frac{(k-j+1)^{p}}{ {b_j}^p}.$$
Combining the estimates and applying the rightmost estimate from Lemma~\ref{low-est}, we get
\begin{align*}
\mathbb E\, \,  \sum _{j=1}^{k} \, \jm_{1\leq i\leq n}|x_{i}\xi_{i}|^p
&\leq B_p\sum _{j=1}^{m-1}  \frac{(k-j+1)^{p}}{ {b_j}^p}+
2^p \, B_p \, \max _{1\leq j\leq k}  \frac{ (k-j+1)^{1+p}}{ {b_j}^p }\\
&\leq B_p\big(1+2\cdot 4^p\big)\sum _{j=1}^{k}\frac{(k-j+1)^{p}}{ {b_j}^p},
\end{align*}
and the proof is complete.
\end{proof}

Finally, we formulate the comparison theorem
for sums of order statistics (the second part of the theorem below was stated in the introduction
as Theorem~\ref{t: comparison intro}).

\begin{theo}\label{comparison}
Let $p, \alpha,\beta>0$, $\delta\in (0,1)$ and $A>1$. Let $1\leq k \leq n$
and $0<x_1\leq \ldots\leq x_n$.
Further, let $\xi_i$, $\eta_i$, $i\leq n$, be random variables satisfying the $\abc$-condition and
condition \eqref{cdf-decay} with parameters $\delta$ and $A$. Assume in addition that
$\xi_i$, $i\leq n$, are jointly independent. Then
$$
 \mathbb E\, \,  \sum _{j=1}^{k} \, \jm_{1\leq i\leq n}|x_{i}\xi_{i}|^p  \leq
  6\, \left(\frac{32 A \alpha}{\delta \beta } \right)^p \, \Gamma(1+p)\,
  \mathbb E\, \,  \sum _{j=1}^{k} \, \jm_{1\leq i\leq n}|x_{i}\eta_{i}|^p  .
$$
In particular, if $\xi_i,\eta_i$ are standard Gaussian variables then
$$
\mathbb E\, \,  \sum _{j=1}^{k} \, \jm_{1\leq i\leq n}|x_{i}\xi_{i}|^p  \leq
6\, \left(C p \right)^p \,
\mathbb E\, \,  \sum _{j=1}^{k} \, \jm_{1\leq i\leq n}|x_{i}\eta_{i}|^p.
$$
where $C>0$ is an absolute constant.
\end{theo}
\begin{proof}
In view of Theorem~\ref{dep-median}, we have
\begin{align*}
\mathbb E\, \,  \sum _{j=1}^{k} \, \jm_{1\leq i\leq n}|x_{i}\eta_{i}|^p
&\geq \frac{1}{2}\sum _{j=1}^{k} \, \bigg({\rm Med}\Big(\jm_{1\leq i\leq n}|x_{i}\eta_{i}|\Big)\bigg)^p\\
&\geq
\frac{1}{2}\frac{\delta^p}{(2 A \alpha)^p}\sum_{\ell=1}^k
\max _{1\leq j\leq \ell}  \frac{(\ell-j+1)^p}{ {b_j}^p }.
\end{align*}
Hence, by Lemma~\ref{low-est} we get
$$\mathbb E\, \,  \sum _{j=1}^{k} \, \jm_{1\leq i\leq n}|x_{i}\eta_{i}|^p
\geq \frac{\delta^p}{2(8A \alpha)^p}\sum _{j=1}^{k}  \frac{(k-j+1)^{p}}{ {b_j}^p}.$$
It remains to apply Theorem~\ref{mainn}.
\end{proof}

\section{Proof of Theorem~\ref{mainMZ}}
\label{proof}

In \cite{MZ} it was shown that Theorem~\ref{comparison} implies Theorem~\ref{mainMZ}.
For the sake of completeness we outline the proof here.

Note that for every sequence $(z_i)_{i=1}^n$ and every permutation $\sigma$ of $\{1, ..., n\}$ one has 
\begin{equation}\label{perm}
 \sum _{j=1}^{k} \jm _{i\leq n} z_i  =  \sum _{j=1}^{k}   \jm _{i\leq n} z_{\sigma(i)}.
\end{equation}

Let $\bar T=(\bar t_{ij})_{ij}$ be an orthogonal transformation of $\R^n$, $X=(X_1, \ldots, X_n)$
be a centered Gaussian vector with independent components and set $Y=(Y_1, \ldots, Y_n):=TX$.
Fix any $k<n$.
For each $i\leq n$, denote the variance of $X_i$ by $\bar a_i$ and the variance of $Y_i$  by $\bar b_i$. 
By $(a_i)_{i\leq n}$ and $(b_i)_{i\leq n}$ we denote the non-increasing rearrangements of $(\bar a_i)_{i\leq n}$ 
and $(\bar b_i)_{i\leq n}$, and let $\sigma$ and  $\pi$ be permutations of $\{1, ..., n\}$ such that 
$a_i=\bar a_{\sigma(i)}$ and $b_i=\bar b_{\pi(i)}$ for all $i\leq n$. By (\ref{perm}) we have 
$$
   \E \sum _{j=1}^{k}   \jm _{i\leq n} X_i^2  =  \E \sum _{j=1}^{k}   \jm _{i\leq n} X_{\sigma(i)}^2
   \quad \mbox{ and } \quad
    \E \sum _{j=1}^{k}   \jm _{i\leq n} Y_i^2  =  \E \sum _{j=1}^{k}   \jm _{i\leq n} Y_{\pi(i)}^2 .
$$
For $i,j\leq n$ denote $t_{ij}=\bar t_{\pi(i) \sigma(j)}$ and $T=(t_{ij})_{ij}$, that is, the matrix $T$ is obtained from 
$\bar T$ by multiplying it by permutation matrices corresponding to $\sigma$ and $\pi$. Clearly, $T$ is also orthogonal. 
Since the coordinates of $X$ are independent, for every $i\leq n$ we have
$$
  b_i =  \sum _{j=1}^n t_{ij}^2 a_j.
$$
As $T$ is an orthogonal matrix,
$\sum _{i=1}^n a_i = \sum _{i=1}^n b_i$. Now we show that for every $\ell <n$ one has
$$
   \sum _{i=1}^\ell a_i \geq  \sum _{i=1}^\ell b_i.
$$
First note that the case $\ell=1$ follows by the orthogonality of $T$ and because
$(a_i)_{i\leq n}$ is non-increasing. For $\ell>1$, again using the orthogonality of $T$ and monotonicity
of $(a_i)_i$, we obtain
\begin{align*}
  \sum _{i=1}^{\ell} b_i =\sum _{i=1}^{\ell} \sum _{j=1}^n t_{ij}^2 a_j &=
  \sum _{j=1}^{\ell -1} \sum _{i=1}^{\ell} t_{ij}^2 a_j + \sum _{j=\ell }^n \sum _{i=1}^{\ell} t_{ij}^2 a_j \\
  &\leq \sum _{j=1}^{\ell -1} \sum _{i=1}^{\ell} t_{ij}^2 a_j
  + a_\ell \sum _{j=\ell }^n \sum _{i=1}^{\ell} t_{ij}^2 \\
   &= \sum _{j=1}^{\ell-1}  a_j + \sum _{j=1}^{\ell -1} a_j \bigg(\sum _{i=1}^{\ell} t_{ij}^2-1\bigg)
   + a_\ell\bigg( \ell - \sum _{j=1}^{\ell -1} \sum _{i=1}^{\ell} t_{ij}^2   \bigg) \\
  &=
   \sum _{j=1}^{\ell}  a_j + \sum _{j=1}^{\ell -1} (a_j-a_\ell) \bigg(\sum _{i=1}^{\ell} t_{ij}^2-1\bigg)\\
  &\leq \sum _{i=1}^{\ell} a_i.
\end{align*}
Note that
$$
  \|(x_1, ..., x_n)\| = \left(\sum _{j=1}^{n-k} \jma _{i\leq n} x_i^2\right)^{1/2} =
   \left(\sum _{j=k+1}^{n} \jm _{i\leq n} x_i^2\right)^{1/2},
$$
defines a norm on $\R^n$ (recall that $\jma$ is $j$th maximum of the corresponding sequence).
Therefore the function
$$
  \varphi (x_1, \ldots, x_n) = \sum _{j=k+1}^{n} \jm _{i\leq n} x_i^2
$$
is convex and thus  Theorem~\ref{MPth} yields
$$
   \E \sum _{j=1}^{k}   \jm _{i\leq n} X_i^2  \leq
   \E \sum _{j=1}^{k}   \jm _{i\leq n} (b_i g_i^2),
$$
where $g_1, \ldots, g_n$ are i.i.d.\ standard Gaussian variables.
Theorem~\ref{comparison} completes the proof.
\kkk

\section{Efficiency of the nonlinear approximation}\label{s: nonlinear efficiency}

In this section, we briefly discuss the following question: How efficient is the nonlinear approximation
based on the largest projections, compared to the linear approximation with respect to the same basis?
In what follows, we fix the dimension $n$.
Given a centered random vector $X$ with a well defined covariance matrix (that is, each component of $X$
has a bounded variance), denote by
$\mathcal E(X,m)$ the mean square error of the nonlinear approximation based on $m$
largest projections onto the standard basis vectors, i.e.\
\begin{equation*}\label{eq: nonlin mean sq error}
  \mathcal E(X,m):=\E\sum\limits_{j=1}^{n-m} \jm _{i\leq n}{X_{i}}^2,\quad m< n.
\end{equation*}
Further, we define corresponding error for the linear approximation as
\begin{equation*}\label{eq: lin mean sq error}
  \mathcal E_0(X,m):=\min\limits_{|J|=n-m}\E\sum\limits_{i\in J}{X_{i}}^2  = \sum\limits_{j=1}^{n-m}
  \jm _{i\leq n} \left(\E{X_{i}}^2\right),
\end{equation*}
where the minimum is taken over all subsets of $\{1,2,\dots,n\}$ of cardinality $n-m$.
Obviously, we have
\begin{equation}\label{eq: obvious lin vs nonlin}
\mathcal E(X,m)\leq \mathcal E_0(X,m)
\end{equation}
for all $m< n$. Moreover, if for a fixed $m$ we define a random Gaussian vector $\widetilde X$ with independent components
and
$$
  \E {\widetilde X_{i}}^2= \left\{
  \begin{array}{ll}
   1, \, \mbox{ if } \, i\leq m+1  \\
   0, \, \mbox{ if } \, i> m+1
   \end{array}
\right.
$$
for all $i\leq n$,
then it can be checked that $\mathcal E_0(\widetilde X,m)=1$
whereas $\mathcal E(\widetilde X,m)\approx m^{-2}$.
Thus, the nonlinear approximation can in some cases be significantly more efficient than the linear approximation
as long as the number of projections is the same. However, as we show below,
some kind of a {\it reverse} inequality for \eqref{eq: obvious lin vs nonlin}
is possible under quite general assumptions on the distribution, if we are allowed to slightly
increase the number of projections for the linear approximation:

\begin{pr}\label{lastprop}
Let $u>0$, $m< n/2$ and let $X$ be a centered random vector in $\R^n$ with a well defined covariance matrix such that
\begin{equation}\label{eq: wrd condition}
 u\E {X_{i}}^2
\leq
\int_{0}^\infty\max\Big(\PP\{{X_i}^2\geq \tau\}-\frac{1}{2},0\Big)\,d\tau,\quad i\leq n.
\end{equation}
Then we have
$$u\mathcal E_0(X,2m)\leq \mathcal E(X,m).$$
\end{pr}

Before proving the proposition, we would like to remark that condition \eqref{eq: wrd condition}
is invariant with respect to scalar multiplication of $X_i$'s.  Note also that it
is satisfied, in particular, for any centered Gaussian random vector with
 $u=1/20$. Moreover, this condition holds with $u=\beta^2/(48\alpha ^2)$
for random variables satisfying the $\abc$-condition. Indeed, for such a variable $\xi$,
denoting by $M$ the median of $\xi^2$, we have
by the $\beta$-condition
$$
  \E \xi^2 = 2\int_0^{\infty} t\PP\left\{|\xi|> t\right\}\, dt\leq
  2\int_0^{\infty} t\exp\left(-\beta t\right)\, dt = \frac{2}{\beta ^2}
$$
and, by the $\alpha$-condition, $M \geq 1/(4\alpha^2)$ and
$$
   \int_0^{M} \PP\left\{\xi^2\geq t\right\}\, dt \geq
  \int_0^{ 1/(4\alpha^2)} \left(1-\alpha \sqrt{t}\right)\, dt +
  \int_{ 1/(4\alpha^2)}^M \frac{1}{2}\, dt =  \frac{1}{24 \alpha^2} + \frac{M}{2},
$$
which implies
$$
  \int_{0}^\infty\max\Big(\PP\{{\xi}^2\geq t\}-\frac{1}{2},0\Big)\,dt =
   \int_0^{M} \PP\left\{\xi^2\geq t\right\}\, dt - \frac{M}{2} \geq \frac{1}{24 \alpha^2}
  \geq \frac{\beta ^2}{48 \alpha^2} \E \xi^2.
$$

\begin{proof}[Proof of Proposition~\ref{lastprop}]
Let $\mathcal I$ be a random subset of $\{1,2,\dots,n\}$ such that
$|\mathcal I|=m$ and
$$\sum\limits_{j=1}^{n-m} \jm _{i\leq n}{X_{i}}^2=\sum\limits_{i\in \mathcal I^c} {X_{i}}^2$$
everywhere on the probability space.
Now, let us distinguish two types of components of $X$: we set
$$I:=\big\{i\leq n:\, \PP\{i\in\mathcal I\}\geq 1/2\big\},\;\mbox{ so that }I^c=\big\{i\leq n:\, \PP\{i\in\mathcal I\}< 1/2\big\}.$$
Obviously, we have
$$\mathcal E(X,m)=\E\sum\limits_{i\in \mathcal I^c} {X_{i}}^2
=\sum\limits_{i=1}^n \E\big({X_{i}}^2\chi_{\{i\in \mathcal I^c\}}\big)
\geq \sum\limits_{i\in I^c} \E\big({X_{i}}^2\chi_{\{i\in \mathcal I^c\}}\big).
$$
Next, observe that for every $\tau> 0$ and every $i\in I^c$,
$$\PP\big\{{X_{i}}^2\chi_{\{i\in \mathcal I^c\}}\geq \tau\big\}
=\PP\big\{{X_{i}}^2\geq \tau\big\}-\PP\big\{{X_{i}}^2\geq\tau\mbox{ and }i\in\mathcal I\big\}\geq
\PP\big\{{X_{i}}^2\geq \tau\big\}-\frac{1}{2}.$$
Hence, in view of condition \eqref{eq: wrd condition},
$$\mathcal E(X,m)
\geq \sum\limits_{i\in I^c}\int_{0}^\infty\max\Big(\PP\{{X_i}^2\geq \tau\}-\frac{1}{2},0\Big)\,d\tau
\geq u\sum\limits_{i\in I^c}\E{X_{i}}^2\geq u\mathcal E_0(X,|I|).$$
On the other hand,
$$
  2m = 2\E |\mathcal I| = 2 \E \sum_{i=1}^{n}\chi_{\{i\in \mathcal I\}}  =
  2 \sum_{i=1}^{n}\PP{\{i\in \mathcal I\}} \geq |I|,
$$
and the proof is complete.
\end{proof}

\begin{rem}
In Proposition~\ref{lastprop} we assumed that $m$ is small compared to $n$,
which is a natural condition in context of signal approximation.
For theoretical reasons, it may be interesting to consider the range $m>n/2$.
One could ask the following question:
Let $m>n/2$ and $k:=n-m$. Does there exist an
absolute constant $C>0$ (not depending on $k$, $n$) such that
$\mathcal E_0(X,n-k/2)\leq C \, \mathcal E(X,n-k)$?
It turns out that this is not true even in the case of the
standard Gaussian random vector. Indeed, a direct computation
shows that $\mathcal E_0(X,n-k/2) = k/2$ (for even $k$), while
$\mathcal E(X,n-k) \approx k^3/n^2$. Thus, the above inequality
cannot be true with an absolute constant for $1\leq k\ll n$.
\end{rem}

\begin{rem}
Note that we were able to obtain a reverse-type inequality for \eqref{eq: obvious lin vs nonlin}
when we agreed to increase the number of one-dimensional
projections for the linear approximation, which could be viewed as increasing of the rank (the dimension)
of the corresponding projection. The idea to slightly lose on the optimality of dimension in order to
gain on other parameters was effectively employed in the study of geometry of high-dimensional
convex bodies (see e.g. \cite{LT, LPT} and references therein).
\end{rem}

\section{Proofs of auxiliary results.}
\label{aux-res-pr}

In this section we provide the proofs of results from Sections~\ref{secone} and \ref{min-bounds} for the sake of
completeness.

\begin{proof}[Proof of Lemma~\ref{aboutmin}]
{\it Case 1: $m=1$,} so that $a_1 \leq b_1/k$. Let $b:=b_1$, $n_0:=0$ and, given any $1\leq \ell \leq k$,
let $n_\ell$ be the largest integer not greater than $n$ such that
$$
\sum_{i=1}^{n_\ell} a_{i} \leq \frac{\ell b}{k} .
$$
Since $b/k \geq a_1 \geq a_2 \geq \ldots \geq a_n$, we have
$0=n_0 < 1 \leq n_1 <n_2< \ldots < n_k =n$.
Define a partition  $(A_\ell)_{\ell \leq k}$ of $\{1, 2, \dots, n\}$ as
$
A_\ell:= \left\{ i\,:\, n_{\ell-1} < i \leq {n_\ell} \right\}.
$
If $a_i\leq \frac{b}{2 k}$ for all $i$ then we set $t=0$. Otherwise,
let $t$ be the largest number in $\{1,2,\dots,n\}$ such that $a_t>\frac{b}{2 k}$.
Then
\begin{itemize}
\item[{[i]}]
for every $1\leq\ell\leq k$ such that $n_\ell \leq t$ we have
$
   \sum_{i\in A_\ell} a_{i} \geq a_{n_\ell} > \frac{b}{2 k} ;
$
\item[{[ii]}]
for every $\ell<k$ such that $n_\ell > t$ we have
$
   \sum_{i\in A_\ell} a_{i} \geq \frac{b}{2 k}
$
(otherwise, since $a_{n_\ell +1} \leq \frac{b}{2 k}$, we would have
$$
    \sum_{i=1}^{n_\ell +1} a_{i} = \sum_{i=1}^{n_{\ell-1}} a_{i} +
    \sum_{i\in A_\ell} a_{i} + a_{n_\ell +1} < \frac{(\ell-1)b}{k} +
    \frac{b}{2 k} + \frac{b}{2 k} = \frac{\ell b}{k},
$$
which contradicts  the choice of $n_\ell$);
\item[{[iii]}] for $\ell=k$ we have
$\sum_{i\in A_k} a_{i} =  \sum_{i=1}^{n} a_{i} -\sum_{i=1}^{n_{k-1}}a_{i}  \geq \frac{b}{k}.
$
\end{itemize}
This completes the proof of the case $m=1$.

\smallskip

\noindent
{\it Case 2: $m>1$.}
For $1\leq \ell<m$, choose $A_\ell=\{\ell\}$, and
let $(A_\ell)_{\ell=m}^k$ be the partition of $\{m, m+1, \dots, n\}$
into $k+1-m$ sets constructed in the same way as in Case 1.
Then, by the above argument, for every $\ell\geq m$ we have
$$
\sum_{i\in A_\ell} a_{i} \geq \frac{b_m}{2(k+1-m)},
$$
and the proof is complete.
\end{proof}

\begin{proof}[Proof of Lemma~\ref{lemmin-lower}]
Denote by $A_{k}(t)$ the event $\{|x_{k} \xi_{k}|> t\} =
\{ | \xi_{k}|> t/x_k \} $ and let
$$
A(t):= \{ \min_{k\leq n}|x_{k} \xi_{k}| >t\} = \bigcap _{k\leq n} A_k (t),\quad t>0.
$$
By \eqref{gdistone}, we have
$
\PP\left( A_{k}(t)^{c} \right) \leq \alpha \, t/ x_{k}.
$
Hence,
$$
\PP \left( A(t) \right) \geq 1-\sum_{k=1}^{n} \PP\left( A_{k}(t)^{c} \right) \geq 1- \alpha \,  t \,
  \sum_{k=1}^{n}\ 1/x_{k} = 1- \alpha b\, t,
$$
which proves the first estimate and implies the estimate for the median.
The estimate for the expectation follows by the distribution formula:
$$
 \mathbb E\min_{1\leq i\leq n}|x_{i} \xi_{i}|^p
 = \int_{0}^{\infty}  \PP   \big\{  \min_{1\leq i\leq n} |x_i \xi_{i}| >
  t^{1/p} \big\} dt \geq \int_{0}^{(\alpha b)^{-p}} \left( 1-\alpha b t^{1/p} \right) dt
  = \frac{(\alpha\, b)^{-p}}{1+p}.
$$
\end{proof}

\medskip

\begin{proof}[Proof of Lemma~\ref{lemmin-upper}]
As in the last proof, denote $A_{k}(t):=\{ |x_{k} \xi_{k}|> t\} =
\{| \xi_{k}|> t/x_k \}$ and let $A(t)$ be the intersection of the events. By  \eqref{gdisttwo},
we have
$
\PP\left( A_{k}(t) \right) \leq \exp\left( - \beta t/x_{k}\right).
$
Therefore,
$$
\PP \left( A(t) \right) = \prod_{k=1}^{n} \PP
\left( A_k(t) \right)    \leq   \exp\Big( - \beta \,
\sum_{k=1}^{n} \ t/x_{k} \Big) = \exp\left( - \beta b \, t \right),
$$
which proves the first estimate and implies the estimate for the median.
Again, the bound for the expectation follows by the distribution formula:
$$
 \mathbb E\min_{1\leq i\leq n}|x_{i} \xi_{i}|^p = \int_{0}^{\infty}
 \PP\big\{  \min_{1\leq i\leq n} |x_i \xi_{i} | >
 t^{1/p} \big\}dt \geq \int_{0}^{\infty} \exp\left( - \beta \, b \, t^{1/p} \right)  dt
 = \left( \beta \, b\right)^{-p}  \, p \, \Gamma(p).
$$
\end{proof}

\begin{proof}[Proof of Lemma~\ref{forthtwo}]
Denote $B(t):= \PP\left\{\km_{1\leq i\leq n}|x_{i} \xi_{i}|\leq
 t \right\}$. Clearly, we have
\begin{eqnarray*}
B(t)
&=& \PP\Big\{\exists i_{1},i_2\dots,i_{k}
\leq n :\, | \xi_{i_{j}}|\leq\frac{ t }{x_{i_{j}}}\mbox{ for all }j\leq k
\Big\}   \\
&=& \PP\Big(\bigcup_{\ell=k}^{n}\bigcup_{A\subset \{ 1,\dots,n \}
\atop |A|=\ell} \Big\{| \xi_{i}| \leq
\frac{ t }{x_{i}} \mbox{ for all $i\in A$ $\;$ and }\;| \xi_{i}|>\frac{ t
}{x_{i}}\mbox{ for all $i\notin A$}\Big\}\Big)   \\
&=&
\sum_{\ell=k}^{n} \ \sum_{A\subset \{ 1, \dots, n\} \atop |A|=\ell}\
 \prod_{i \in A}
\PP \Big\{  | \xi_{i}| \leq
\frac{ t }{x_{i}} \Big\} \,
\prod_{i \notin A}
\PP \Big\{|\xi_{i}|>\frac{ t
}{x_{i}}\Big\} .
\end{eqnarray*}
Hence,
$$
B(t) \leq
\sum_{\ell=k}^{n} \sum_{A\subset \{ 1, \dots, n\} \atop |A|=\ell}\
 \prod_{i \in A} \PP \Big\{  | \xi_{i}| \leq
\frac{ t }{x_{i}} \Big\} \,  \leq
\sum_{\ell=k}^{n} \ \sum_{A\subset \{ 1, \dots, n\} \atop |A|=\ell}\
\prod_{i \in A}
\ \frac{\alpha  t }{x_i} .
$$
Corollary~\ref{agmean} implies the first part of the lemma.

Next, we verify the bound for the expectation. The case $k=1$ follows by
Lemma~\ref{lemmin-lower}, so we assume that $k\geq 2$.
Let us start with establishing the bound
\begin{equation}\label{lowgen}
\frac{ k}{2^{1/p}\, 4 \alpha} \,
\left( \sum_{i=1}^n 1/x_i \right) ^{-1}
\leq
\left( \mathbb E \, \,
\km_{1\leq i\leq n} |x_{i}\xi_{i}| ^p \right)^{1/p},\quad k\geq 2.
\end{equation}
Set $\gamma := e/4$. Then
\begin{align*}
\mathbb E\,\km_{1\leq i\leq n}|x_{i} \xi_{i}|^p
&= \int_{0}^{\infty}\PP\big\{  \km_{1\leq i\leq n} |x_i \xi_{i}| >  t^{1/p} \big\} dt\\
&\geq \int_{0}^{(\gamma/a)^{p}} \left( 1-\frac{1}{\sqrt{2 \pi k}} \, \, \frac{a^k t^{k/p}}{1-a t^{1/p}}\right) dt \\
&\geq \left(\frac{\gamma}{a}\right)^{p} - \frac{1}{\sqrt{2\pi k}}\,  \frac{a^k}{1-\gamma}\,\left(\frac{\gamma}{a}\right)^{k+p}\\
&\geq \left(\frac{\gamma}{a}\right)^{p}\left(1 - \frac{1}{2 \sqrt{\pi}}\, \frac{\gamma^2}{1-\gamma} \right)\\
&\geq \frac{1}{2}\, \left(\frac{\gamma}{a}\right)^{p},
\end{align*}
which proves \eqref{lowgen}.
Finally, observe that
for every sequence
$(a_i)_{i=1}^n$ and every $r<k$ one has
$$
  \km (a_i)_{i=1}^n \geq  \krm (a_i)_{i=r+1}^n;
$$
in particular,
$$\left( \mathbb E \, \,
\km_{1\leq i\leq n} |x_{i}\xi_{i}| ^p \right)^{1/p}
\geq \left( \mathbb E \, \,
\krm_{r+1\leq i\leq n} |x_{i}\xi_{i}| ^p \right)^{1/p}
\geq \frac{k-r}{2^{1/p}\, 4 \alpha} \,
\left( \sum_{i=r+1}^n 1/x_i \right) ^{-1},
$$
where the last inequality is \eqref{lowgen} applied to the appropriate ``truncated'' sequence.
The result follows.
\end{proof}

\bigskip

\noindent
{\bf Acknowledgment. } The authors would like to thank Nicole Tomczak-Jaegermann for valuable comments on
Section~\ref{s: nonlinear efficiency}. They are also grateful to anonymous referees for helpful remarks and
suggestions.

\smallskip

\noindent
A. E. Litvak and K. Tikhomirov\\  {\small Dept.\ of Math.\ and Stat.\ Sciences},\\
{\small University of Alberta}, {\small Edmonton, AB, Canada T6G 2G1},\\
{\small \tt  aelitvak@gmail.com}\\
{\small \tt  ktikhomi@ualberta.ca}

\smallskip
\noindent {\small Current address of K.T.: Dept.\ of Math., Fine Hall, Princeton, NJ 08544}


\begin{thebibliography}{123456}
{\footnotesize

\bibitem[ALLPT]{ALLPT}{\sc R. Adamczak, R. Latala, A. Litvak, A. Pajor and N. Tomczak-Jaegermann},
{\em Tail estimates for norms of sums of log-concave random vectors},
Proc. London Math. Soc. 108 (2014), 600--637.
%
\bibitem[DN]{DN}
{\sc  H. A. David,  H. N. Nagaraja},
{\em Order statistics}, 3rd ed.,
Wiley Series in Probability and Statistics.
Chichester: John Wiley \& Sons,  2003.
%
\bibitem[G]{Gl} {\sc E. D. Gluskin},
{\em Extremal properties of orthogonal parallelepipeds
and their applications to the geometry of Banach spaces},
Math. USSR Sbornik, 64 (1989), 85--96.
%
\bibitem[GLSW1]{GLSW}
{\sc Y. Gordon, A. E. Litvak, C. Sch\"utt, E. Werner},
{\em Orlicz Norms of Sequences of Random Variables},
Ann. of Prob., 30 (2002), 1833--1853.
%
\bibitem[GLSW2]{GLSW-Bull}
{\sc Y. Gordon, A. E. Litvak, C. Sch\"utt, E. Werner},
{\em Geometry of spaces between zonoids and polytopes},
Bull. Sci. Math., 126 (2002), 733--762.
%
\bibitem[GLSW3]{GLSW-CRAS}
{\sc Y. Gordon, A. E. Litvak, C. Sch\"utt, E. Werner},
{\em Minima of sequences of Gaussian random variables},
 C. R. Acad. Sci. Paris, S\'er. I Math.,
340 (2005), 445--448.
%
\bibitem[GLSW4]{GLSW-PAMS}
{\sc Y. Gordon, A. E. Litvak, C. Sch\"utt, E. Werner},
{\em On the minimum of several random variables},
Proc. Amer. Math. Soc. 134 (2006), 3665--3675.
%
\bibitem[GLSW5]{GLSW-P}
{\sc Y. Gordon, A. E. Litvak, C. Sch\"utt, E. Werner},
{\em Uniform estimates for order statistics and Orlicz functions},
Positivity, 16 (2012), 1--28.
%
\bibitem[HLP]{HLP}
{\sc G. H. Hardy, J. E. Littlewood and G. Polya},
{\em Inequalities}, 2nd ed.,  Cambridge,
The University Press. XII, 1952.
%
\bibitem[Ho]{H56}
{
{\sc W. Hoeffding},
{\em On the distribution of the number of successes in independent trials},
Ann. Math. Statist. 27 (1956), 713--721.
}
%
\bibitem[L]{L-JFA} {\sc R. Latala},
{\em Order statistics and concentration of $\ell_r$ norms for log-concave vectors},
J. Funct. Anal. 261 (2011), 681--696.
%
\bibitem[LPP]{LPP} {\sc R. Lechner, M. Passenbrunner, J. Prochno},
{\em Uniform estimates for averages of order statistics of matrices},
Electron. Commun. Probab. 20 (2015), no. 27, 1--12.
%
\bibitem[LPT]{LPT} {\sc A.E. Litvak, A. Pajor, N. Tomczak-Jaegermann},
{\em Diameters of Sections and Coverings of Convex Bodies,}
J. of Funct. Anal., 231 (2006), 438--457.
%
\bibitem[LT]{LT} {\sc A.E. Litvak, N. Tomczak-Jaegermann},
{\em Random aspects of high-dimensional convex bodies,}
GAFA, Lecture Notes in Math., 1745, 169--190, Springer-Verlag, 2000.
%
\bibitem[M]{M} {\sc S. Mallat},
{\em A wavelet tour of signal processing. The Sparse Way}, 3rd edition, Academic Press, 2008.
%
\bibitem[MZ]{MZ} {\sc S. Mallat, O. Zeitouni},
{\em A conjecture concerning optimality of the Karhunen-Loeve basis
in nonlinear reconstruction},
arXiv:1109.0489.
%
\bibitem[MP]{MP} {\sc A.W.~Marshall, F.~Proschan},
{\em An Inequality for Convex Functions Involving Majorization},
J. Math. Anal. Appl.
12 (1965), 87--90.
%
\bibitem[MS]{MS}
{\sc S. Montgomery-Smith}, {\em Rearrangement invariant norms of
symmetric sequence norms of independent sequences of random variables},
Isr. J. Math. 131 (2002), 51--60.
%
%
%
%
\bibitem[Sid]{Si} {\sc Z. \u{S}id\'{a}k},
{\em Rectangular confidence regions for the means of multivariate
normal distributions}, J. Am. Stat. Assoc. 62 (1967), 626--633.
%
\bibitem[Z]{Z} {\sc O. Zeitouni},
{\em A correlation inequality for nonlinear reconstruction},
Workshop on the Mathematical Foundations of Learning Theory, Paris 2006,
www.diffusion.ens.fr/index.php?idconf=1436\&res=conf.
}

\end{thebibliography}
\end{document}